\documentclass[11pt]{article}
\usepackage{amsmath,amsthm,epsfig,amsfonts,color}
 \usepackage{rotating}

 \usepackage[round]{natbib} 
\usepackage{graphics, graphicx, hyperref, url, amsmath, subfigure, amsthm, amsfonts, enumerate, epstopdf, booktabs, bm}

\allowdisplaybreaks
\def\singlespace{\def\baselinestretch{1.5}\@normalsize}

\newtheorem{theorem}{{Theorem}}
\newtheorem{lemma}{{Lemma}}
\newtheorem{remark}{{Remark}}
\renewcommand{\baselinestretch}{1.5}

\hypersetup{
    pdfstartview={FitH},    
    colorlinks=true,        
    linkcolor=blue,         
    citecolor=blue,         
    filecolor=blue,         
    urlcolor=blue           
}

\def\marginnote#1{\setbox0=\vtop{\hsize4pc
\small\raggedright\noindent\baselineskip9pt \rightskip=0.5pc plus
1.5pc #1}\leavevmode \vadjust{\dimen0=\dp0
\kern-\ht0\hbox{\kern-4.00pc\box0}\kern-\dimen0}}
\def\lboxit#1{\vbox{\hrule\hbox{\vrule\kern6pt
\vbox{\kern6pt#1\kern6pt}\kern6pt\vrule}\hrule}}

\def\be{\begin{eqnarray}}
\def\en{\end{eqnarray}}
\def\bes{\begin{eqnarray*}}
\def\ens{\end{eqnarray*}}
\def\toind{\overset{d}{\longrightarrow}}
\def\toinp{\overset{p}{\longrightarrow}}

\begin{document}
\thispagestyle{empty}
\begin{center}
{\Large \textbf{Limit theory for an AR(1) model with intercept and a possible infinite variance}}
\end{center}
\vskip 10 pt \centerline{\sc  Qing Liu$^{1,2}$ and Xiaohui Liu$^{1,2}$}
\begin{footnotetext}
{\hspace*{-0.25 in}
$^1$School of Statistics, Jiangxi University of Finance and Economics, Nanchang, Jiangxi 330013, China\\
$^2$Research Center of Applied Statistics, Jiangxi University of Finance and Economics, Nanchang, Jiangxi 330013, China
}
\end{footnotetext}

{\bf Abstract.} In this paper, we derive the limit distribution of the least squares estimator for an AR(1) model with a non-zero intercept and a possible infinite variance. It turns out that the estimator has a quite different limit for the cases of $|\rho| < 1$, $|\rho| > 1$, and $\rho = 1 + \frac{c}{n^\alpha}$ for some constant $c \in R$ and $\alpha \in (0, 1]$, and whether or not the variance of the model errors is infinite also has a great impact on both the convergence rate and the limit distribution of the estimator.



\section{Introduction}

As a simple but useful tool, the auto-regression (AR) models have been widely used in economics and many other fields. Among them, the most simplest one is the autoregressive process with order 1, i.e., an AR(1) model, which is usually defined as
\begin{equation}\label{mod:1}
  y_t = \mu + \rho y_{t-1} + e_t, \quad t=1,\cdots,n.
\end{equation}
where $y_0$ is a constant and $e_t$'s are identically and independent distributed random errors with mean zero and finite variance. The process $\{y_t\}$ is i) stationary if $|\rho|<1$ independent of $n$, ii) unit root if $\rho=1$, iii) near unit root if $\rho=1+c/n$ for some nonzero constant $c$, iv) explosive if $|\rho| > 1$ independent of $n$, and v) moderate deviation from a unit root if $\rho=1+c/k_n$ for some nonzero constant $c$ and a sequence $\{k_n\}$ satisfying $k_n\to\infty$ and $k_n/n\to 0$ as $n\to\infty$.

When $\mu = 0$ and the error variance in model \eqref{mod:1} is finite, it is well known in the literature that the least squares estimator for $\rho$ has a quite different limit distribution in cases of stationary, unit root and near unit root; see \cite{Phi1987}. The convergence rate of the correlation coefficient is $\sqrt{n}$, $n$ for cases i)-iii), respectively, and may even be $(1 + c)^n$ in the case of v) for some $c > 0$ as stated in \cite{PM2007}. More studies on this model can be found in \cite{DF1981}, \cite{DR2006}, \cite{Mik2007}, \cite{AG2009, AG2014}, \cite{SS1999}, \cite{CLP2012}, \cite{HLP2016} and references therein.

When $\mu \neq 0$ with finite variance, \cite{WY2015} and \cite{Fei2018} studied the limit theory for the AR(1) for cases of iv) and v), respectively. It is shown that the inclusion of a nonzero intercept may change drastically the large sample properties of the least squares estimator compared to \cite{PM2007}. More recently, \cite{LP2017} studied how to construct a uniform confidence region for $(\mu, \rho)$ regardless of i)-v) based on the empirical likelihood method.

Observe that $e_t$ may have an infinite variance in practice \citep{Phi1990, CT2009}, and most of the aforementioned researches were focused on the case of $e_t$ having a finite variance. In this paper, we consider model \eqref{mod:1} when $\mu \neq 0$ and the variance of $e_t$ may possibly be infinite. We will derive the limit distribution of the least squares estimator of $(\mu, \rho)$ for the following cases:
\begin{itemize}
    \item P1) $|\rho|<1$ independent of $n$;
    \item P2) $|\rho| > 1$ independent of $n$;
    \item P3) $\rho = 1$;
    \item P4) $\rho = 1 + \frac{c}{n}$ for some constant $c \neq 0$;
    \item P5) $\rho = 1 + \frac{c}{n^\alpha}$ for some constants $c < 0$ and $\alpha \in (0, 1)$;
    \item P6) $\rho = 1 + \frac{c}{n^\alpha}$ for some constants $c > 0$ and $\alpha \in (0, 1)$.
\end{itemize}
Since the current paper allows for the inclusion of both the intercept and a possible infinite variance, it can be treated as an extension of the existing literature, i.e., \cite{Phi1987, PM2007, HPW2014, WY2015, Fei2018}, among others.

We organize the rest of this paper as follows. Section 2 provides the methodology and main limit results. Detailed proofs are put in Section 3.

\section{Methodology and main results}

Under model (\ref{mod:1}), by minimizing the sum of squares:
\[
    \sum_{t=1}^n (y_t-\mu-\rho y_{t-1})^2,
\]
with respect to $(\mu, \rho)^\top$, we get the least squares estimator for $(\mu,\rho)^\top$ as follows
\begin{equation}
\left\{\begin{array}{ll}
\hat\mu=\frac{\sum\limits_{s=1}^ny_s\sum\limits_{t=1}^n y_{t-1}^2-\sum\limits_{s=1}^ny_{s-1}\sum\limits_{t=1}^n y_ty_{t-1}}{n\sum\limits_{t=1}^n y_{t-1}^2-\left(\sum\limits_{t=1}^n y_{t-1}\right)^2}\\[4ex]
\hat\rho=\frac{n\sum\limits_{t=1}^n y_ty_{t-1}-\sum\limits_{s=1}^ny_{s-1}\sum\limits_{t=1}^n y_t}{n\sum\limits_{t=1}^n y_{t-1}^2- \left(\sum\limits_{t=1}^n y_{t-1}\right)^2}.
\end{array}\right.
\end{equation}
Here $A^\top$ denotes the transpose of the matrix or vector $A$. In the sequel we will investigate the limit distribution of $(\hat\mu-\mu, \hat\rho-\rho)^\top.$

To derive the limit distribution of this least squares estimator, we follow \cite{PM2007} by assuming that
\begin{itemize}
    \item C1) The innovations $\{e_t\}$ are iid with $E[e_t]=0$;
    \item C2) The process is initialized at $y_0=O_p(1)$.
\end{itemize}
Observing that the variance of $e_t$'s may not exist, we use the slowly varying function $l(x)=E[e_t^2I(|e_t|\leq x)]$ instead as did in \cite{HPW2014} to characterize the dispersion property of the random errors, which is supposed to satisfy
\begin{itemize}
  \item C3) $l(nx)/l(n) \to 1$ as $n \to \infty$ for any $x > 0$.
\end{itemize}
An example of slowly varying function is when $l(x)$ has a limit, say $\lim\limits_{x\to\infty}l(x)=\sigma^2$, which implies $\{e_t\}$ having a finite variance $\sigma^2$. Another example is $l(x)=\log(x)$, $x>1$, which implies that the variance of $e_t$'s does not exist. One known property of $l(x)$ is that $l(x)=o(x^{\varepsilon})$ as $x\to\infty$, for any $\varepsilon>0$. More properties on $l(x)$ can be found in \cite{EKM1997}. To deal with the possibly infinite variance, we introcude the following sequence $\{b_k\}_{k=0}^\infty$, where
\[b_0=\inf\{ x\geq 1: l(x)>0\}\]
and
\[b_j=\inf\left\{s: s\geq b_0+1, \frac{l(s)}{s^2}\leq \frac1j\right\},~\text{for}~j=1,2,\ldots,n,\]
which imply directly $n l(b_n)\leq b_n^2$ for all $n\geq 1$; see also \cite{GGM1997}.

For convenience, in the sequel we still call $|\phi| < 1$ the stationary case, $\rho = 1$ the unit root case, $\phi = 1 + \frac{c}{n}$ for some $c\neq 0$ the near unit root case, $\rho = 1 + \frac{c}{n^\alpha}$ for some $c \neq 0$ the moderate deviation case and $|\rho| > 1$ the explosive case, even when the variance of $v_t$ is infinite. We will divide the theoretical derivations into four separate subsections.

\subsection{Limit theory for the stationary case}

We first consider the stationary case $|\rho| < 1$, which is independent of $n$. Observe that
\bes
  y_t = \mu + \rho y_{t-1} + e_t = \frac{\mu}{1 - \rho} + \left(y_0 - \frac{\mu}{1 - \rho}\right) \rho^t + \sum_{j=1}^{t} \rho^{t-j} e_j.
\ens
We write $\bar{y}_t = y_t - \frac{\mu}{1 - \rho}$, and then have
\[\bar y_t=\rho \bar y_{t-1}+e_t.\]

To prove the main result for this case, we need the following preliminary lemma.
\begin{lemma}\label{lem:01}
Suppose conditions C1)-C3) hold. Under P1), as $n\to\infty$, we have
\bes
    &&\frac1n \sum_{t=1}^n y_{t-1} \overset{p}{\longrightarrow} \frac\mu{1-\rho},\\
    &&\frac1{nl(b_n)} \sum_{t=1}^n y_{t-1}^2  \overset{p}{\longrightarrow}
    \begin{cases}
        \frac1{1-\rho^2},&~\text{if}~\lim\limits_{m\to\infty}l(b_m)=\infty,\\
        \frac1{1-\rho^2}+\frac{\mu^2}{\sigma^2(1-\rho)^2},&~\text{if}~\lim\limits_{m\to\infty}l(b_m)=\sigma^2,
    \end{cases}
\ens
and
\be
\label{eqn:lem1part3}
  \begin{pmatrix}
    \frac1{\sqrt{nl(b_n)}}\sum\limits_{t=1}^n e_t\\[2ex]
    \frac1{\sqrt{n}l(b_n)}\sum\limits_{t=1}^n \bar y_{t-1}e_t
  \end{pmatrix}
   \overset{d}{\longrightarrow}
  \begin{pmatrix}
    W_1\\
    W_2
  \end{pmatrix} \sim N\left(\begin{pmatrix}0\\0
    \end{pmatrix},
    \begin{pmatrix}1\quad0\\~0\quad\frac1{1-\rho^2}
    \end{pmatrix}
    \right).
\en
\end{lemma}

Based on Lemma \ref{lem:01}, we can show the following theorem.

\begin{theorem}\label{th:01}
Under conditions C1)-C3), as $n\to\infty$, we have under P1) that
    \[\left(\begin{array}{ccc}
          \sqrt{\frac{n}{l(b_n)}}(\hat\mu-\mu) \\
          \sqrt{n} (\hat\rho-\rho)
        \end{array}
        \right)
   \overset{d}{\longrightarrow} \left( \begin{array}{ccc}
          X_1 \\
          X_2
        \end{array}
        \right); \]
where $X_1= W_1 - \frac{\mu(1+\rho)}{\sigma}W_2$ and $X_2 = (1-\rho^2)W_2$ if $\lim\limits_{m \to \infty} l(b_m) = \sigma^2$, and $X_1 = W_1$ and $X_2 = (1-\rho^2)W_2$ if $\lim\limits_{m \to \infty} l(b_m) = \infty$.
\end{theorem}

\begin{remark}
  Theorem \ref{th:01} indicates that the possible infinite variance may affect both the convergence rate and the limit distribution of the least squares estimator of the intercept, but has no impact on those of $\rho$.
\end{remark}

\begin{remark}
  When $\lim\limits_{m\to \infty} l(b_m)$ exists and is equal to $\sigma^2$, we have $(X_1, X_2)^\top \sim N(0, \Sigma_1)$, where $\Sigma_1 = (\sigma_{ij}^2)_{1\le i,j\le 2}$ with $\sigma_{11}^2 = 1 + \frac{\mu^2(1 + \rho)}{\sigma^4 (1 - \rho)}$, $\sigma_{12}^2 = \sigma_{21}^2 = - \frac{\mu(1+\rho)}{\sigma^2}$ and $\sigma_{22}^2 = 1 - \rho^2$. That is, the limit distribution reduces to the ordinary case; see \cite{LP2017} and references therein for details.
\end{remark}

\subsection{Limit theory for the explosive case}

For this case, let $\tilde{y}_t=\sum_{i=1}^t \rho^{t-i}e_i+\rho^t y_0$, then
\bes
    y_t=\mu\frac{1-\rho^t}{1-\rho}+\rho^t y_0+\sum_{i=1}^t \rho^{t-i}e_i=\mu\frac{1-\rho^t}{1-\rho}+\tilde{y}_t.
\ens

Along the same line as Section 2.1, we derive a preliminary lemma first as follows.

\begin{lemma} \label{lem:02}
Suppose conditions C1)-C3) hold. Under P2), as $n\to\infty$, we have
\[
\begin{pmatrix}
\frac1{\sqrt{nl(b_n)}}\sum\limits_{t=1}^n e_t\\[2ex]
\frac1{\sqrt{l(b_n)}}\sum\limits_{t=1}^n\rho^{-(n-t)}e_t\\[2ex]
\frac\rho{\sqrt{l(b_n)}}\sum\limits_{t=1}^{n-1}\rho^{-t}e_t+\rho y_0
\end{pmatrix}
\overset{d}{\longrightarrow}
\begin{pmatrix}
W_1\\
U_1\\
U_2
\end{pmatrix},\]
and
\[
\begin{pmatrix}
\rho^{-(n-2)}\{l(b_n)\}^{-1}\sum\limits_{t=1}^n \tilde y_{t-1}e_t\\[2ex]
(\rho^2-1)\rho^{-2(n-1)}\{l(b_n)\}^{-1}\sum\limits_{t=1}^n \tilde y_{t-1}^2
\end{pmatrix}
\overset{d}{\longrightarrow}
\begin{pmatrix}
U_1U_2\\[2ex]
U_2^2
\end{pmatrix},\]
where $U_1 \sim \lim\limits_{m\to \infty} \frac1{\sqrt{l(b_m)}}\sum\limits_{t=1}^m\rho^{-(m-t)}e_t$, $U_2 \sim\rho y_0 + \lim\limits_{m\to \infty} \frac\rho{\sqrt{l(b_m)}}\sum\limits_{t=1}^{m-1}\rho^{-t}e_t$, and $W_1$ are mutually independent random variables. $W_1$ is specified in Lemma \ref{lem:01}.
\end{lemma}

Using this lemma, we can obtain the following theorem.

\begin{theorem}\label{th:explosive}
Under conditions C1)-C3), as $n\to\infty$, we have
    \[\left(\begin{array}{ccc}
          \sqrt{\frac n {l(b_n)}}(\hat\mu-\mu) \\
          \rho^n(\hat\rho-\rho)
        \end{array}
        \right)
   \overset{d}{\longrightarrow}
   \left( \begin{array}{ccc}
          W_1 \\
          (\rho^2-1)\frac{U_1}{U_2+\mu\rho/(\rho-1)}
        \end{array}
        \right), \]
   under P2).
\end{theorem}

\begin{remark}
Similar to the case with finite variance, the least square estimator of the intercept is asymptotically normal, regardless of the error distribution. While the limit distribution of $\hat\rho$ depends on the distribution of $e_t$'s, hence no invariance principle is applicable.
\end{remark}

\begin{remark}
Theorem \ref{th:explosive} indicates that the possible infinite variance only affects the convergence rate of $\hat\mu$. The joint limit distribution reduces to that obtained in \cite{WY2015} if $\lim\limits_{m\to\infty} l(b_m)$ is finite.
\end{remark}

\subsection{Limit theory for the unit root and near unit root cases}

In these two cases, $\rho=:\rho_n=1+\frac cn, c\in \mathbb{R}$. ($c = 0$ corresponds to $\rho = 1$, i.e., the unit root case.) Let $\tilde{y}_t=\sum_{i=1}^t \rho^{t-i}e_i$, then
\bes
    y_t=\mu+\rho y_{t-1}+e_t=\mu\left(\sum_{j=0}^{t-1}\rho^j\right)+\rho^t y_0+\tilde{y}_t.
\ens
We have the following Lemma.

\begin{lemma}\label{lem:unit}
1) Let $E_n(t)=\frac{\sum\limits_{i=1}^{[ns]}e_i}{\sqrt{nl(b_n)}}$, $s\in[0, 1]$. Then
\[E_n(s) \overset{D}{\longrightarrow} \widetilde W(s),~\text{in}~ D[0, 1]~\text{as}~n\to\infty,\]
where $\{\widetilde W(s), s\ge 0\}$ is a standard Brownian process, $[\cdot]$ is the floor function, and $\overset{D}{\longrightarrow}$ denotes the weak convergence. Moreover, define $J_c(s) = \lim\limits_{a \to c} \frac{1-e^{as}}{-a}$, then as $n\to\infty$, we have under P3) and P4) that
\[\left\{\begin{array}{ll}
&n^{-2}\sum\limits_{t=2}^n\left(\sum\limits_{j=0}^{t-2}\rho^j\right)\rightarrow \int_0^1 J_{c}(s)\, ds,\\[3.5ex]
&n^{-3}\sum\limits_{t=2}^n\left(\sum\limits_{j=0}^{t-2}\rho^j\right)^2\rightarrow \int_0^1 J_{c}^2(s)\, ds,\\[3.5ex]
&n^{-3/2}\sum\limits_{t=2}^n\left(\sum\limits_{j=0}^{t-2}\rho^j\right) \frac{e_t}{\sqrt{l(b_n)}} \overset{d}{\longrightarrow} \int_0^1 J_{c}(s)\, d\widetilde W(s).
\end{array}\right.\]
and in turn
\[\frac{\sum\limits_{k=1}^n\tilde{y}_t^2}{n^2 l(b_n)} \overset{d}{\longrightarrow}  \int_0^1 e^{-2c(1-s)} \widetilde W^2(B_c(s))\,ds,\]
\[\frac{\sum\limits_{k=1}^n\tilde{y}_t}{n^{3/2} \sqrt{l(b_n)}}\overset{d}{\longrightarrow} \int_0^1 e^{-c(1-s)} \widetilde W(B_c(s))\,ds,\]
\[\frac{\sum\limits_{k=1}^n\tilde{y}_{t-1}e_t}{n l(b_n)} \overset{d}{\longrightarrow}  -c\int_0^1 e^{-2c(1-s)}\widetilde W^2(B_c(s))\,ds\
+\frac{\widetilde W^2(B_c(1))}2-\frac12,\]
where $B_c(s)=e^{2c}(e^{-2cs-1})/(-2c)$.
\end{lemma}

This lemma can be showed easily by using similar techniques as in \cite{CW1987} based on the fact that
\bes
    \frac1{n^{3/2}\sqrt{l(b_n)}}\sum_{t=2}^n\left(\sum_{j=0}^{t-2}\rho^j\right)e_t^{(2)} \overset{p}{\longrightarrow} 0,~n\to\infty.
\ens
Using this lemma, it is easy to check the following theorem.


\begin{theorem}\label{th:unit}
Under the conditions C1)-C3), as $n\to\infty$, we have
    \[\left(\begin{array}{ccc}
          \sqrt{\frac{n}{l(b_n)}} (\hat\mu-\mu) \\[2ex]
          \sqrt{\frac{n^3}{l(b_n)}} (\hat\rho-\rho)
        \end{array}
        \right)
   \overset{d}{\longrightarrow}
   \left( \begin{array}{ccc}
          Y_1/d \\
          Y_2/(\mu d)
        \end{array}
        \right); \]
under P3) and P4), where
\bes
    d&=&\int_0^1 J_c^2(s)\,ds -\left(\int_0^1 J_c(s)\,ds\right)^2,\\[2ex]
    Y_1&=&\widetilde W(1) \int_0^1 J_c^2(s)\,ds - \int_0^1 J_c(s)\,ds \int_0^1 J_{c}(s)\, d\widetilde W(s),\\[2ex]
    Y_2&=&\int_0^1 J_{c}(s)\, d\widetilde W(s)- \widetilde W(1)\int_0^1 J_c(s)\,ds.
\ens
\end{theorem}

\subsection{Limit theory for the moderate deviation cases}

As stated in \cite{PM2007}, the moderate deviation cases bridge the different convergence rates of cases P1)-P4). That is, the case P5) bridges the stationary case and the near unit root case, while Case P6) bridges the explosive case and the near unit root case. And the derivation of these two cases need to be handled differently, because for the $c > 0$ case central limit theorem for martingales fails to hold. Following \cite{PM2007}, we consider them separately.

The following lemma is useful in deriving the limit distribution of the least square estimator under cases P5)-P6).

\begin{lemma} \label{lem:Munit}
Suppose conditions C1)-C3) hold.
\begin{itemize}
  \item i) Under P5), as $n \to \infty$, we have
\begin{eqnarray*}
  \begin{pmatrix}
    \frac1{\sqrt{n}}\sum\limits_{t=1}^n\frac{e_t}{\sqrt{l(b_n)}}\\
    \frac1{\sqrt{n^\alpha}}\sum\limits_{t=1}^n\frac{\rho^{t-1}e_t}{\sqrt{l(b_n)}}\\
    \frac1{\sqrt{n^\alpha}}\sum\limits_{t=1}^n\frac{\rho^{n-t}e_t}{\sqrt{l(b_n)}}\\
    \frac1{\sqrt{n^{1+\alpha}}}\sum\limits_{t=2}^n \frac{e_t}{\sqrt{l(b_n)}} \sum\limits_{j=1}^{t-1} \frac{\rho^{t-1-j}e_j}{\sqrt{l(b_n)}}
  \end{pmatrix}
  ~\overset{d}{\longrightarrow}~
    \begin{pmatrix}
      V_{11}\\
      V_{12}\\
      V_{13}\\
      V_{14}
    \end{pmatrix} \sim N(0, \Sigma_2),
\end{eqnarray*}
where $\Sigma_2= diag (1, -\frac1{2c},-\frac1{2c},-\frac1{2c})$, which implies that $V_{1i}$'s are independent;

\item ii) Under P6), as $n \to \infty$, we have
\begin{eqnarray*}
  \begin{pmatrix}
    \frac1{\sqrt{n}}\sum\limits_{t=1}^n\frac{e_t}{\sqrt{l(b_n)}}\\
    \frac1{\sqrt{n^\alpha}}\sum\limits_{t=1}^n\frac{\rho^{-t}e_t}{\sqrt{l(b_n)}}\\
    \frac1{\sqrt{n^{\alpha}}}\sum\limits_{t=1}^n\frac{\rho^{t-1-n}e_t}{\sqrt{l(b_n)}}
  \end{pmatrix}
  ~\overset{d}{\longrightarrow}~
    \begin{pmatrix}
      V_{21}\\
      V_{22}\\
      V_{23}
    \end{pmatrix},
\end{eqnarray*}
and
\[
    \frac1{\rho^n n^\alpha}\sum_{t=2}^n \left(\sum_{i=1}^{t-1}\frac{\rho^{t-1-i}e_i}{\sqrt{l(b_n)}}\right)\frac{e_t}{\sqrt{l(b_n)}} \overset{d}{\longrightarrow} V_{22}V_{23},
\]
where $(V_{21}, V_{22}, V_{23})^\top\sim N(0, \Sigma_3)~\text{with}~\Sigma_3=diag(1, \frac1{2c},\frac1{2c})$, which implies that $V_{2i}$'s are independent.
\end{itemize}
\end{lemma}

\begin{theorem}\label{th:Munit}
Suppose conditions C1)-C3) hold.
\begin{itemize}
\item[] 1) Under P5), we have as $n\to\infty$
    \[\left(\begin{array}{ccc}
          a_n(\hat\mu-\mu) \\
          a_nn^\alpha (\hat\rho-\rho)
        \end{array}
        \right)
   \toind
   \left( \begin{array}{ccc}
          \frac{\mu}{cd} \\
          \frac{1}{d}
        \end{array}
        \right)^\top Z; \]
where
\begin{eqnarray*}
  a_n &=& \sqrt{n^\alpha/l(b_n)}I(\alpha>1/2) + \sqrt{n^{1-\alpha}}I(\alpha \le 1/2),\\
  Z &=& \frac{\mu}{c} V_{12} I(\alpha > 1/2) + V_{14} I(\alpha \le 1/2),\\
  d &=& \frac{\mu^2}{-2c^3} I(\alpha > 1/2) + \frac{1}{-2c}I(\alpha \le 1/2),
\end{eqnarray*}
if $\lim\limits_{m\to \infty} l(b_m) = \infty$, and
\begin{eqnarray*}
  a_n &=& n^{\max(\alpha, 1/2) - \alpha/2},\\
  Z &=& \frac{\mu\sigma}{c} V_{12} I(\alpha \ge 1/2) + \sigma^2 V_{14} I(\alpha \le 1/2),\\
  d &=& \frac{\mu^2}{-2c^3} I(\alpha \ge 1/2) + \frac{\sigma^2}{-2c}I(\alpha \le 1/2),
\end{eqnarray*}
if $\lim\limits_{m\to \infty} l(b_m) = \sigma^2$.

\item[] 2) Under P6), we have as $n\to\infty$
\[\left(\begin{array}{ccc}
          \sqrt{\frac n{l(b_n)}}(\hat\mu-\mu) \\
          \sqrt{\frac {n^{3\alpha}}{l(b_n)}}\rho^n(\hat\rho-\rho)
        \end{array}
        \right)
   \toind
   \left( \begin{array}{ccc}
          V_{21} \\
         \frac{2c^2}{\mu}V_{23}
        \end{array}
        \right). \]
\end{itemize}
\end{theorem}

\begin{remark}
  Similar to \cite{Fei2018}, the least squares estimators under P6) are asymptotically normal, in contrast to the Cauchy distribution in \cite{PM2007}.
   Moreover, the joint limit distribution is still degenerated under P5), but slightly differently, we obtain the exact limit distribution in this case.
\end{remark}

\begin{remark}
As can be seen from Theorem \ref{th:Munit}, under P5), the possible infinite variance has no impact on the asymptotic behavior of estimators when $\alpha< \frac12$, but affects the convergence rate when $\alpha>\frac12$, and the limit distribution when $\alpha=\frac12$.
\end{remark}

\begin{remark}
  Under some mild conditions, it is possible to extend the current result under P6) into to the cases that $\rho = 1 + \frac{c}{k_n}$ by using similar arguments for some general sequence $\{k_n\}$ such that $k_n = o(n)$ and $k_n / n \to 0$ as $n \to \infty$ as studied in \cite{PM2007}. But it is impossible to do such an extension under P5) because the derivation of the limit distribution involving the order comparison between $\sqrt{n}$ and $k_n$, while the limit of $\sqrt{n} / k_n$ as $n\to\infty$ is unclear without any further information of $k_n$.
\end{remark}

\section{Detailed proofs of the main results}

In this section, we provide all detailed proofs of the lemmas and theorems stated in Section 2.

\begin{proof}[Proof of Lemma \ref{lem:01}]
To handle the possible infinite invariance, we use the truncated random variables. Let
\begin{equation}
\left\{\begin{array}{ll}
e_t^{(1)}=e_tI(|e_t|\leq b_n)- E[e_tI(|e_t|\leq b_n)],\\[2ex]
e_t^{(2)}=e_tI(|e_t|> b_n)- E[e_tI(|e_t|> b_n)],
\end{array}\right.
\end{equation}
where $I(\cdot)$ denotes the indicative function. The key step is to show that the difference of replacing $e_t$ by $e_t^{(1)}$ in the summations is negligible.

Let $\{\bar y_t^{(1)}\}$ and  $\{\bar y_t^{(2)}\}$ be two time series satisfying
\bes
    \bar y_t^{(k)}=\rho \bar y_{t-1}^{(k)}+ e_t^{(k)}, \quad k=1,2.
\ens
Obviously, $\{e_t^{(1)} / \sqrt{l(b_n)}: t\geq1\}$ are iid, and under P1), it is easy to check that $\{\bar y_{t-1}^{(1)}e_t^{(1)} / l(b_n) : t\geq1\}$ is a martingale differences sequence with respect to $\mathcal{F}_{t-1} =\sigma(\{e_s: s \leq t-1\})$ for $t = 1, 2, \cdots, n$, which satisfy the Lindeberg condition. Hence, by the Cram\'{e}r-Wold device and the central limit theorem for martingale difference sequences, we have
\be
\label{eqn:et10}
    \begin{pmatrix}
    \frac1{\sqrt{nl(b_n)}}\sum\limits_{t=1}^n e_t^{(1)}\\[2ex]
    \frac1{\sqrt{n}l(b_n)}\sum\limits_{t=1}^n \bar y_{t-1}^{(1)} e_t^{(1)}
    \end{pmatrix}
    \overset{d}{\longrightarrow}
    \begin{pmatrix}
    W_1\\
    W_2
    \end{pmatrix}.
\en

Next, under condition C3), it follows from \cite{CSW2003} that
\bes
    E[|e_t|I(|e_t|>b_n)]=o(l(b_n)b_n^{-1}),\quad n\to\infty.
\ens
Then by $nl(b_n)\leq b_n^2$ and the Markov inequality, we have, for any $\varepsilon > 0$,
\begin{eqnarray*}
  P\left(\left|\sum_{t=1}^n e_t^{(2)}\right| \ge \sqrt{nl(b_n)}\varepsilon\right) &\leq& \frac{\sum\limits_{t=1}^n E|e_t^{(2)}|}{\sqrt{nl(b_n)}\varepsilon}\\
  &\leq& 2\frac{\sum\limits_{t=1}^nE[|e_t|I(|e_t|>b_n)]}{\sqrt{nl(b_n)}\varepsilon}\\
  &=& o\left(\frac{\sqrt{nl(b_n)}}{b_n}\right) = o(1), ~ \text{ as } n \to \infty,
\end{eqnarray*}
That is,
\be
\label{eqn:et11}
    \frac1{\sqrt{nl(b_n)}}\sum_{t=1}^n e_t^{(2)}  = o_p(1),\quad n\to\infty,
\en

Furthermore, note that $\bar y_{t-1}^{(k)}=\rho^{t-1}\bar y_0^{(k)}+\sum_{i=1}^{t-1}\rho^{t-1-i}e_i^{(k)},~k=1,2$. By the H\"{o}lder inequality, we have
\[E\left(\left|\frac{e_t^{(1)}}{\sqrt{l(b_n)}}\right| \right) \leq \left\{E\left(\frac{e_t^{(1)}}{\sqrt{l(b_n)}}\right)^2\right\}^{1/2}\leq 1.\]
Using the Markov inequality, we have
\bes
&&P\left(\left|\sum_{t=1}^n \bar y_{t-1}^{(1)}e_t^{(2)}\right| \geq \sqrt{n}l(b_n)\varepsilon\right)\\
&&\leq \frac{1}{\sqrt{n}l(b_n)\varepsilon}\sum_{t=1}^nE(|\bar y_{t-1}^{(1)}|)E(|e_t^{(2)}|)\\
&&\leq \frac{1}{\sqrt{n}l(b_n)\varepsilon}\frac1{1-\rho}\{E(|\bar y_0^{(1)}|)+n E(|e_t^{(1)}|)\}E(|e_t^{(2)}|)\\
&&= o(n^{-1/2}b_n^{-1}) +  o\left(\frac{\sqrt{nl(b_n)}}{b_n}\right)=o(1),
\ens
i.e.,
\be
\label{eqn:et121}
    \frac1{\sqrt{n}l(b_n)}\sum_{t=1}^n \bar y_{t-1}^{(1)}e_t^{(2)} = o_p(1), \quad \text{as } n \to \infty.
\en

Similarly we can show
\be\label{eqn:et1223}
    \frac1{\sqrt{n}l(b_n)}\sum_{t=1}^n \bar y_{t-1}^{(2)}e_t^{(j)} = o_p(1),\quad n\to\infty, ~ j = 1, 2.
\en
This, together with \eqref{eqn:et11}-\eqref{eqn:et121}, shows
\bes
  \begin{pmatrix}
    \frac1{\sqrt{nl(b_n)}}\sum\limits_{t=1}^n e_t\\[2ex]
    \frac1{\sqrt{n}l(b_n)}\sum\limits_{t=1}^n \bar y_{t-1}e_t
  \end{pmatrix}
  =
  \begin{pmatrix}
    \frac1{\sqrt{nl(b_n)}}\sum\limits_{t=1}^n e_t^{(1)}\\[2ex]
    \frac1{\sqrt{n}l(b_n)}\sum\limits_{t=1}^n \bar y_{t-1}^{(1)}e_t^{(1)}
  \end{pmatrix}
  + o_p(1), \quad \text{as } n \to \infty,
\ens
while combined with \eqref{eqn:et10} shows \eqref{eqn:lem1part3}.

Note that $\bar y_t=\rho \bar y_{t-1}+e_t$. Multiplying both sides with $\bar y_t$ and $\bar y_{t-1}$ respectively, and taking summation, we have
\bes
\begin{cases}
  \sum\limits_{t=1}^n \bar y_t^2=\rho\sum\limits_{t=1}^n \bar y_t\bar y_{t-1}+\sum\limits_{t=1}^n \bar y_te_t,\\[3ex]
  \sum\limits_{t=1}^n \bar y_t\bar y_{t-1}=\rho\sum\limits_{t=1}^n \bar y_{t-1}^2+\sum\limits_{t=1}^n \bar y_{t-1}e_t.
\end{cases}
\ens
Since
\[\frac1{\sqrt{n}l(b_n)}\sum_{t=1}^n \bar y_{t-1}e_t \toind W_2,\]
and
\[\sum_{t=1}^n \bar y_te_t=\rho\sum_{t=1}^n \bar y_{t-1}e_t+\sum_{t=1}^n e_t^2,\]
we have
\[\frac{\sum_{t=1}^n \bar y_te_t}{nl(b_n)} \toinp 1,~n\to\infty,\]
by noting that $\frac{\sum_{t=1}^n e_t^2}{nl(b_n)} \toinp 1$ (see (3.4) in \cite{GGM1997}). Hence,
\bes
\left\{\begin{array}{ll}
\frac{\sum\limits_{t=1}^n \bar y_t^2}{nl(b_n)}=\rho\frac{\sum\limits_{t=1}^n \bar y_t\bar y_{t-1}}{nl(b_n)}+1+o_p(1)\\[2ex]
\frac{\sum\limits_{t=1}^n \bar y_t\bar y_{t-1}}{nl(b_n)}=\rho\frac{\sum\limits_{t=1}^n \bar y_{t-1}^2}{nl(b_n)}+o_p(1),
\end{array}\right.
\ens
which implies that as $n\to\infty$
\bes
\left\{\begin{array}{ll}
\frac{\sum\limits_{t=1}^n \bar y_t^2}{nl(b_n)} \toinp \frac1{1-\rho^2}\\[2ex]
\frac{\sum\limits_{t=1}^n \bar y_t\bar y_{t-1}}{nl(b_n)} \toinp \frac\rho{1-\rho^2}.
\end{array}\right.
\ens

Note that
\bes
    \frac1n \sum_{t=1}^n y_{t-1} &=& \frac1n \sum_{t=1}^n \bar y_{t-1}+\frac\mu{1-\rho} = \frac\mu{1-\rho} + o_p(1),
\ens
and
\[
    \frac1{nl(b_n)} \sum_{t=1}^n y_{t-1}^2=\frac1{nl(b_n)} \sum_{t=1}^n \bar y_{t-1}^2
    +\frac{2\mu}{1-\rho}\frac1{nl(b_n)} \sum_{t=1}^n \bar y_{t-1}+\frac1{l(b_n)}\frac{\mu^2}{(1-\rho)^2}.
\]
Using these, the rest proof of this lemma follows directly by the law of large numbers.
\end{proof}

\begin{proof}[Proof of Theorem \ref{th:01}]
For the least squares estimator, it is easy to check that
\[
\begin{pmatrix}
\hat\mu-\mu\\
\hat\rho-\rho
\end{pmatrix}
=
\begin{pmatrix}
\frac{\sum\limits_{t=1}^ny_{t-1}^2\sum\limits_{s=1}^ne_s -\sum\limits_{s=1}^ny_{s-1} \sum\limits_{t=1}^ne_ty_{t-1}}{n\sum\limits_{t=1}^ny_{t-1}^2-\left(\sum\limits_{t=1}^ny_{t-1}\right)^2}\\[4ex]
\frac{n\sum\limits_{t=1}^ny_{t-1}e_t-\sum\limits_{t=1}^ny_{t-1}\sum\limits_{s=1}^ne_s}{n\sum\limits_{t=1}^ny_{t-1}^2 -\left(\sum\limits_{t=1}^ny_{t-1}\right)^2}
\end{pmatrix}.
\]

For convenience, hereafter write
\begin{eqnarray*}
  \Delta_1 &=& \sum_{t=1}^ny_{t-1}^2\sum_{s=1}^ne_s -\sum_{s=1}^ny_{s-1} \sum_{t=1}^ny_{t-1}e_t,\\
  \Delta_2 &=& n\sum_{t=1}^ny_{t-1}e_t-\sum_{t=1}^ny_{t-1}\sum_{s=1}^ne_s,\\
  \Delta_3 &=& n\sum_{t=1}^ny_{t-1}^2-\left(\sum_{t=1}^ny_{t-1}\right)^2.
\end{eqnarray*}
Observe that
\bes
\Delta_1&=&\sum_{t=1}^ny_{t-1}^2\sum_{s=1}^n e_s -\sum_{s=1}^ny_{s-1} \sum_{t=1}^ny_{t-1} e_t\\
&=& \left(\sum_{t=1}^ny_{t-1}^2 - \frac\mu{1-\rho}\sum_{t=1}^n y_{t-1}\right) \sum_{s=1}^n e_s-\sum_{s=1}^ny_{s-1}\sum_{t=1}^n\bar y_{t-1} e_t.
\ens
Hence, by Lemma \ref{lem:01} we have, as $n\to\infty$,
\bes
    \frac{1}{(nl(b_n))^{3/2}} \Delta_1 \overset{d}{\longrightarrow} \frac1{1-\rho^2} W_1 - \frac\mu{\sigma^2(1-\rho)}W_2 I\left(\lim\limits_{m\to \infty} l(b_m) = \sigma^2\right).
\ens
Next, relying on
\bes
\Delta_2&=&n\sum_{t=1}^ny_{t-1}e_t-\sum_{t=1}^ny_{t-1}\sum_{s=1}^ne_s\\
&=& n\sum_{t=1}^n\bar y_{t-1}e_t-\sum_{t=1}^n\bar y_{t-1}\sum_{s=1}^ne_s\\
&=& n\left(\sum_{t=1}^n\bar y_{t-1}e_t\right)\{1+o_p(1)\},
\ens
we obtain
\[\frac1{n^{3/2}l(b_n)} \Delta_2 \overset{d}{\longrightarrow} W_2.\]
Following a similar fashion, we have
\bes
    \frac{1}{n^2l(b_n)}\Delta_3 = \frac{1}{n l(b_n)} \sum_{t=1}^n y_{t-1}^2 - \frac1{l(b_n)} \left(\frac1n\sum_{t=1}^ny_{t-1}\right)^2 \overset{p}{\longrightarrow} \frac{1}{1 - \rho^2}.
\ens
Then this theorem follows immediately by using Slutsky's theorem.
\end{proof}

\begin{proof}[Proof of Lemma \ref{lem:02}]
For the first part, by following a similar fashion to Lemma \ref{lem:01}, we can show that
\[
\begin{pmatrix}
\frac1{\sqrt{nl(b_n)}}\sum\limits_{t=1}^n e_t\\[2ex]
\frac1{\sqrt{l(b_n)}}\sum\limits_{t=1}^n\rho^{-(n-t)}e_t\\[2ex]
\frac\rho{\sqrt{l(b_n)}}\sum\limits_{t=1}^{n-1}\rho^{-t}e_t+\rho y_0
\end{pmatrix} =
\begin{pmatrix}
\frac1{\sqrt{nl(b_n)}}\sum\limits_{t=1}^n e_t^{(1)}\\[2ex]
\frac1{\sqrt{l(b_n)}}\sum\limits_{t=1}^n\rho^{-(n-t)}e_t^{(1)}\\[2ex]
\frac\rho{\sqrt{l(b_n)}}\sum\limits_{t=1}^{n-1}\rho^{-t}e_t^{(1)}+\rho y_0
\end{pmatrix} + o_p(1).\]
The rest proof is similar to \cite{And1959} and \cite{WY2015}. We omit the details.

For the second part, we only prove the case of $\lim\limits_{m\to\infty}l(b_m)=\infty$. Let $\tilde{y}_t^{(k)}=\sum\limits_{i=1}^t \rho^{t-i}e_i^{(k)}+\rho^t y_0$, $k=1,2$, $t = 1, 2, \cdots, n$.
Similar to Lemma \ref{lem:01}, it is easy to verify that
\bes
    \rho^{-n}\{l(b_n)\}^{-1}\sum_{t=1}^n E|\tilde y_{t-1}^{(k)}e_t^{(j)}| \overset{p}{\longrightarrow}  0,\\
    \rho^{-2n}\{l(b_n)\}^{-1}\sum_{t=1}^n E|\tilde y_{t-1}^{(k)}\tilde y_{t-1}^{(j)}| \overset{p}{\longrightarrow} 0,
\ens
for $(k,j) \in \{(1, 2), (2, 1), (2, 2)\}$, as $n\to\infty$, and in turn we can obtain that
\[
\begin{pmatrix}
\rho^{-(n-2)}\{l(b_n)\}^{-1}\sum\limits_{t=1}^n \tilde y_{t-1}e_t\\[2ex]
(\rho^2-1)\rho^{-2(n-1)}\{l(b_n)\}^{-1}\sum\limits_{t=1}^n \tilde y_{t-1}^2
\end{pmatrix} =
\begin{pmatrix}
\rho^{-(n-2)}\{l(b_n)\}^{-1}\sum\limits_{t=1}^n \tilde y_{t-1}^{(1)}e_t^{(1)}\\[2ex]
(\rho^2-1)\rho^{-2(n-1)}\{l(b_n)\}^{-1}\sum\limits_{t=1}^n (\tilde y_{t-1}^{(1)})^2
\end{pmatrix} + o_p(1).\]
Then the conclusion follows.
\end{proof}

\begin{proof}[Proof of Theorem \ref{th:explosive}]
Using the same arguments as \cite{WY2015}, it follows from Lemma \ref{lem:02} that
\bes
\begin{cases}
    \rho^{-(n-1)}\{l(b_n)\}^{-1/2} y_n \overset{d}{\longrightarrow}  U_2+\frac{\mu\rho}{\rho-1},\\[2ex]
    \rho^{-(n-2)}\{l(b_n)\}^{-1}\sum\limits_{t=1}^n y_{t-1}e_t \overset{d}{\longrightarrow} U_1(U_2+\frac{\mu\rho}{\rho-1}),\\[2ex]
    (\rho-1)\rho^{-(n-1)}\{l(b_n)\}^{-1/2}\sum\limits_{t=1}^n y_{t-1} \overset{d}{\longrightarrow}  U_2+\frac{\mu\rho}{\rho-1},\\[2ex]
    (\rho^2-1)\rho^{-2(n-1)}\{l(b_n)\}^{-1}\sum\limits_{t=1}^n y_{t-1}^2 \overset{d}{\longrightarrow}  (U_2+\frac{\mu\rho}{\rho-1})^2.
\end{cases}
\ens
Then as $n\to\infty$, we have
\bes
    \frac1{n^{1/2}\rho^{2n}\{l(b_n)\}^{3/2}}\Delta_1 &=& \rho^{-2n}\{l(b_n)\}^{-1}\sum_{t=1}^n y_{t-1}^2\times\frac1{\sqrt{nl(b_n)}}\sum_{t=1}^n e_t+o_p(1)\\[2ex]
    &\overset{d}{\longrightarrow}& \frac1{\rho^2(\rho^2-1)} W_1 \left(U_2+\frac{\mu\rho}{\rho-1}\right)^2\\[3ex]
    \frac1{n\rho^{n}l(b_n)}\Delta_2&=&\rho^{-n}\{l(b_n)\}^{-1}\sum_{t=1}^n y_{t-1}e_t+o_p(1)\\[2ex]
    &\overset{d}{\longrightarrow}& \frac1{\rho^2}U_1 \left(U_2+\frac{\mu\rho}{\rho-1}\right).
\ens
and
\bes
    \frac1{n\rho^{2n}l(b_n)}\Delta_3 &= &\rho^{-2n}\{l(b_n)\}^{-1}\sum_{t=1}^n y_{t-1}^2+o_p(1)\\[2ex]
    &\overset{d}{\longrightarrow}& \frac1{\rho^2(\rho^2-1)}\left(U_2+\frac{\mu\rho}{\rho-1}\right)^2.
\ens
Then the theorem has been proved.
\end{proof}

\begin{proof}[Proof of Theorem \ref{th:unit}]
Similar to the proof of Theorem \ref{th:01}, by Lemma \ref{lem:unit}, we have, as $n \to \infty$,
\bes
    \Delta_1&=&\sum_{t=1}^ny_{t-1}^2\sum_{t=1}^ne_t -\sum_{t=1}^ny_{t-1} \sum_{t=1}^ny_{t-1}e_t\\
    &=&\mu^2 \left\{\sum_{t=2}^n \left(\sum_{j=0}^{t-2}\rho^j\right)^2\sum_{s=1}^ne_s - \left(\sum_{s=2}^n \sum_{j=0}^{s-2}\rho^j\right) \left(\sum_{t=2}^n\sum_{j=0}^{t-2}\rho^je_t \right) \right\}\{1+o_p(1)\},
\ens
which implies
\[\frac{1}{\sqrt{n^7 l(b_n)}} \Delta_1 \overset{d}{\longrightarrow} \mu^2\left(\widetilde W(1)\int_0^1 J_c^2(s)\,ds-\int_0^1 J_c(s)\,ds \int_0^1 J_{c}(s)\, d\widetilde W(s)\right),\]
and
\bes
\Delta_2&=&n\sum_{t=1}^ny_{t-1}e_t-\sum_{t=1}^ny_{t-1}\sum_{t=1}^ne_t\\
&=&\mu \left\{n\left(\sum_{t=2}^n\sum_{j=0}^{t-2}\rho^je_t\right) - \left(\sum_{t=2}^n \sum_{j=0}^{t-2}\rho^j\right)\sum_{t=1}^ne_t \right\}\{1+o_p(1)\},
\ens
which leads
\[\frac{1}{\sqrt{n^5 l(b_n)}} \Delta_2 \overset{d}{\longrightarrow} \mu\left(\int_0^1 J_{c}(s)\, d\widetilde W(s) - \widetilde W(1)\int_0^1 J_c(s)\,ds\right),\]
and
\bes
\Delta_3&=&n\sum_{t=1}^ny_{t-1}^2- \left(\sum_{t=1}^ny_{t-1}\right)^2\\
&=&\mu^2 n\sum_{t=2}^n \left(\sum_{j=0}^{t-2}\rho^j \right)^2-\mu^2 \left(\sum_{t=2}^n \sum_{j=0}^{t-2}\rho^j \right)^2+o_p(n^4),
\ens
which results in
\[\frac{1}{n^4} \Delta_3 \rightarrow \mu^2\left(\int_0^1 J_c^2(s)\,ds -\left(\int_0^1 J_c(s)\,ds \right)^2\right),\quad n\to\infty.\]
Then the theorem has been proved.
\end{proof}

\begin{proof}[Proof of Lemma \ref{lem:Munit}]
i) Similar to Lemma \ref{lem:01}, under P5), by the Markov inequality and the fact $nl(b_n)\leq b_n^2$, we have for any $\varepsilon > 0$
\begin{eqnarray*}
P\left(\left|\sum_{t=1}^n \rho^{t-1}e_t^{(2)}\right| \geq \sqrt{n^\alpha l(b_n)}\varepsilon\right)
&\leq& \frac{\sum_{t=1}^n \rho^{t-1}E|e_t^{(2)}|}{\sqrt{n^\alpha l(b_n)}\varepsilon}\\
&\leq& 2\frac{\sum_{t=1}^n\rho^{t-1}E[|e_t|I(|e_t|>b_n)]}{\sqrt{n^\alpha l(b_n)}\varepsilon}\\
&=& o\left(\frac{l(b_n)}{b_n}\right)\frac{1 - \rho^{n}}{1-\rho}\frac1{\sqrt{n^\alpha l(b_n)}\varepsilon}\\
&=& o\left(\frac{\sqrt{n^\alpha l(b_n)}}{b_n}\right) \to 0, ~ \text{ as } n \to \infty.
\end{eqnarray*}
This implies that
\[\frac1{\sqrt{n^\alpha}}\sum_{t=1}^n\frac{\rho^{t-1}e_t^{(2)}}{\sqrt{l(b_n)}} \toinp 0 ~ \text{ as } n \to \infty.\]
Similarly we can show that
\[\frac1{\sqrt{n}}\sum_{t=1}^n\frac{e_t^{(2)}}{\sqrt{l(b_n)}} \toinp 0\quad\text{and}
\quad\frac1{\sqrt{n^\alpha}}\sum_{t=1}^n\frac{\rho^{n-t}e_t^{(2)}}{\sqrt{l(b_n)}} \toinp 0,~ n\to\infty.\]
Next, if $(i,j)\in\{(1,2), (2,1), (2,2)\}$, it follows from Lemma A.2 of \cite{HPW2014} that
\[\frac1{\sqrt{n^{1+\alpha}}}\sum_{t=2}^n \left\{\sum_{k=1}^{t-1}\frac{\rho^{t-1-k}e_k^{(i)}}{\sqrt{l(b_n)}} \right\} \frac{e_t^{(j)}}{\sqrt{l(b_n)}} \toinp 0, \quad n\to\infty.\]
We actually obtain
\begin{eqnarray}
\label{eqn:tranc}
  \begin{pmatrix}
    \frac1{\sqrt{n}}\sum\limits_{t=1}^n\frac{e_t}{\sqrt{l(b_n)}}\\
    \frac1{\sqrt{n^\alpha}}\sum\limits_{t=1}^n\frac{\rho^{t-1}e_t}{\sqrt{l(b_n)}}\\
    \frac1{\sqrt{n^\alpha}}\sum\limits_{t=1}^n\frac{\rho^{n-t}e_t}{\sqrt{l(b_n)}}\\
    \frac1{\sqrt{n^{1+\alpha}}}\sum\limits_{t=2}^n \frac{e_t}{\sqrt{l(b_n)}} \sum\limits_{j=1}^{t-1} \frac{\rho^{t-1-j}e_j}{\sqrt{l(b_n)}}
  \end{pmatrix}
  ~=~
  \begin{pmatrix}
    \frac1{\sqrt{n}}\sum\limits_{t=1}^n\frac{e_t^{(1)}}{\sqrt{l(b_n)}}\\
    \frac1{\sqrt{n^\alpha}}\sum\limits_{t=1}^n\frac{\rho^{t-1}e_t^{(1)}}{\sqrt{l(b_n)}}\\
    \frac1{\sqrt{n^\alpha}}\sum\limits_{t=1}^n\frac{\rho^{n-t}e_t^{(1)}}{\sqrt{l(b_n)}}\\
    \frac1{\sqrt{n^{1+\alpha}}}\sum\limits_{t=2}^n \frac{e_t^{(1)}}{\sqrt{l(b_n)}} \sum\limits_{j=1}^{t-1} \frac{\rho^{t-1-j}e_j^{(1)}}{\sqrt{l(b_n)}}
  \end{pmatrix} + o_p(1).
\end{eqnarray}

Then, based on the Cram\'{e}r-Wold device and central limit theorem for martingales differences sequence, the Lindeberg condition for the first part of the right side of \eqref{eqn:tranc} can be proved by using the same arguments as \cite{PM2007} and \cite{HPW2014}. We omit the details.

ii) The proof of the case under P6) is similar to that of i) and \cite{PM2007}, thus is omitted.
\end{proof}

\begin{proof}[Proof of Theorem \ref{th:Munit}]
i) Under P5), observe that $\rho^n = o(n^{-\alpha})$, $y_0 = O_p(1)$, and
\begin{eqnarray*}
    y_t&=& \mu + \rho y_{t-1} + e_t\\
     &=& \frac{\mu}{1 - \rho} + \left(\frac \mu c n^\alpha+y_0 \right)\rho^t+\sum_{i=1}^t\rho^{t-i}e_i,
\end{eqnarray*}
which implies
\begin{eqnarray*}
y_n&=& \left(\frac{\mu}{1 - \rho} + \sum_{i=1}^n\rho^{n-i}e_i\right) \{1 +o_p(1)\}.
\end{eqnarray*}
Note that $y_t-y_{t-1}=\mu+(\rho-1)y_{t-1}+e_t$ and
\[y_t^2-y_{t-1}^2=\mu^2+(\rho^2-1)y_{t-1}^2+e_t^2+2\mu\rho y_{t-1}+2\mu e_t+2\rho y_{t-1}e_t,\]
it is easy to verify that
\begin{eqnarray*}
  \sum_{t=1}^n y_{t-1} &=& \frac{1}{1 - \rho} \Big(n\mu - y_n + \sum_{t=1}^n e_t\Big)\{1 + o_p(1)\},\\
  \sum_{t=1}^n y_{t-1}^2 &=& \frac{1}{1 - \rho^2} \Big(n\mu^2 - y_n^2+y_0^2 + \sum_{t=1}^n e_t^2 + 2\mu\rho \sum_{t=1}^n y_{t-1} + 2\mu \sum_{t=1}^n e_t + 2 \rho \sum_{t=1}^n y_{t-1}e_t\Big)\\
  &=& \frac{1}{1 - \rho^2} \Big(n\mu^2 - y_n^2 + \sum_{t=1}^n e_t^2 + 2\mu\rho \sum_{t=1}^n y_{t-1}\Big)\{1 + o_p(1)\},
\end{eqnarray*}
and
\begin{eqnarray*}
\sum_{t=1}^n y_{t-1}e_t &=& \left\{\sum_{t=1}^n  \left(\frac{\mu}{1 - \rho} + \sum_{i=1}^{t-1}\rho^{t-1-i}e_i\right)e_t\right\}\{1 + o_p(1)\}.
\end{eqnarray*}
Hence,
\begin{eqnarray*}
    \Delta_1& = &\sum\limits_{t = 1}^n y_{t-1}^2 \cdot \sum\limits_{t = 1}^n e_{t}  -\sum\limits_{t = 1}^n y_{t-1} \cdot \sum\limits_{t = 1}^n y_{t-1} e_t\\
    &=& \left\{\frac{1}{1 - \rho^2} \Big(n\mu^2 - y_n^2 + \sum_{t=1}^n e_t^2 + 2\mu\rho \sum_{t=1}^n y_{t-1} + 2\mu \sum_{t=1}^n e_t + 2 \rho \sum_{t=1}^n y_{t-1}e_t\Big) \cdot \sum\limits_{t = 1}^n e_{t} \right.\\
    &&\quad \left. - \frac{1}{1 - \rho} \Big(n\mu - y_n + \sum_{t=1}^n e_t\Big) \cdot \sum\limits_{t = 1}^n y_{t-1} e_t \right\} \{1 + o_p(1)\}\\
    &=&\frac{1}{1 - \rho} \Big\{\Big(\frac{n\mu^2 - y_n^2 + \sum_{t=1}^n e_t^2}{2} + \mu\rho \sum_{t=1}^n y_{t-1} + \mu \sum_{t=1}^n e_t\Big) \cdot \sum\limits_{t = 1}^n e_{t}  - \Big(n\mu - y_n\Big) \times\\
    &&\quad \Big[\frac{\mu}{1 - \rho} \sum_{t=1}^n e_t - \frac{\mu}{1 - \rho} \sum_{t=1}^n \rho^{t-1}e_t + \sum_{t=1}^n \Big(e_t \sum_{j=0}^{t-1} \rho^{t-1-j}e_j\Big)\Big]\Big\}\{1  + o_p(1)\}\\
    &=&\frac{n\mu}{1 - \rho} \Big\{\frac{\mu}{1 - \rho} \sum_{t=1}^n \rho^{t-1}e_t - \sum_{t=1}^n \Big(e_t \sum_{j=0}^{t-1} \rho^{t-1-j}e_j\Big)\Big\}\{1  + o_p(1)\}\\
    &=& \left\{n^{1+5\alpha/2}\sqrt{l(b_n)} \frac{\mu^2}{c^2} \left(\frac{1}{\sqrt{n^\alpha}} \sum_{t=1}^n \rho^{t-1} \frac{e_t}{\sqrt{l(b_n)}}\right) + \right. \\
    && \left. n^{3(1+\alpha)/2}l(b_n) \frac{\mu}{c} \left(\frac{1}{\sqrt{n^{1 + \alpha}}} \sum_{t=1}^n \frac{e_t}{\sqrt{l(b_n)}} \sum_{j=0}^{t-1} \rho^{t-1-j}\frac{e_j}{\sqrt{l(b_n)}}\right)\right\}\{1 + o_p(1)\},
\end{eqnarray*}
and
\begin{eqnarray*}
    \Delta_2&=&-\sum\limits_{t = 1}^n y_{t-1} \cdot \sum\limits_{t = 1}^n e_{t} + n \sum\limits_{t = 1}^n y_{t-1} e_t\\
    &=&-\frac{1}{1 - \rho} \Big(n\mu - y_n + \sum_{t=1}^n e_t\Big) \cdot \sum\limits_{t = 1}^n e_{t} + n \Big[\frac{\mu}{1 - \rho} \sum_{t=1}^n e_t\\
    && - \frac{\mu}{1 - \rho} \sum_{t=1}^n \rho^{t-1}e_t + \sum_{t=1}^n \Big(e_t \sum_{j=0}^{t-1} \rho^{t-1-j}e_j\Big)\Big]\\
    &=& \left\{n^{1+3\alpha/2} \sqrt{l(b_n)} \cdot \frac{\mu}{c} \cdot \left(\frac{1}{\sqrt{n^{\alpha}}} \sum_{t=1}^n \rho^{t-1} \frac{e_t}{\sqrt{l(b_n)}} \right) \right. \\
    && + \left. n^{3/2+\alpha/2} l(b_n) \left(\frac{1}{\sqrt{n^{1+\alpha}}} \sum_{t=1}^n\frac{e_t}{\sqrt{l(b_n)}} \sum_{j=0}^{t-1} \rho^{t-1-j}\frac{e_j}{\sqrt{l(b_n)}}\right)\right\}\{1 + o_p(1)\},
\end{eqnarray*}
and
\begin{eqnarray*}
    \Delta_3 & =& n\sum\limits_{t=1}^n y_{t-1}^2 - \Big(\sum\limits_{t=1}^n y_{t-1}\Big)^2\\
    & = &\frac{n}{1 - \rho}\Big\{\frac{1}{1+\rho} \Big[\frac{2 n\mu^2 \rho}{1 - \rho} + \Big(n\mu^2 + \sum_{t=1}^n e_t^2\Big) - y_n^2 - \frac{2\mu\rho}{1 - \rho} y_n + \frac{4\rho \mu}{1 - \rho} \sum_{t=1}^n e_t\Big] \\
    && - \frac{1}{1 - \rho} \Big[n\mu^2 - 2 \mu y_n + 2\mu \sum_{t=1}^n e_t\Big]\Big\}\{1  + o_p(1)\}\\
    & = &\frac{n}{1 - \rho}\Big\{\frac{1}{1+\rho} \Big[\sum_{t=1}^n e_t^2 - y_n^2 - \frac{2\mu\rho}{1 - \rho} y_n\Big] + \frac{2 \mu y_n}{1 - \rho}\Big\}\{1  + o_p(1)\}\\
    & = &\left\{\frac{n^{2+\alpha}l(b_n)}{-2c} \left(\frac{1}{n} \sum_{t=1}^n \frac{e_t^2}{l(b_n)}\right) + n^{1+3\alpha} \frac{\mu^2}{-2c^3}\right\}\{1 + o_p(1)\}.
\end{eqnarray*}
These, together with Lemma \ref{lem:Munit}, lead directly to i).

ii)Under P6), it follows from Lemma \ref{lem:Munit} that
\begin{eqnarray*}
  y_n= \frac\mu c n^\alpha\rho^n+n^{\alpha/2}\rho^n\sqrt{l(b_n)}V_{22} \{1+o_p(1)\},~n \to \infty,
\end{eqnarray*}
which implies that
\begin{eqnarray*}
  y_n^2= \frac{\mu^2}{ c^2} n^{2\alpha}\rho^{2n}+2\frac\mu cn^{3\alpha/2}\rho^{2n}\sqrt{l(b_n)}V_{22}\{1+o_p(1)\}
\end{eqnarray*}
and
\begin{eqnarray*}
  \sum_{t=1}^n y_{t-1}&=&\frac1c n^\alpha y_n-\frac1c n^\alpha y_0-\frac\mu c n^{\alpha+1}-\frac1c n^\alpha\sum_{t=1}^n e_t\\
  &=&\frac\mu{c^2}n^{2\alpha}\rho^n+\frac1c n^{3\alpha/2}\rho^n\sqrt{l(b_n)}V_{22}\{1+o_p(1)\}.
\end{eqnarray*}
Again by Lemma \ref{lem:Munit}, we can show that
\begin{eqnarray*}
\sum_{t=1}^ny_{t-1}e_t&=&\sum_{t=1}^n \left(-\frac\mu c n^\alpha+\frac\mu c n^\alpha\rho^{t-1}+y_0\rho^{t-1}+\sum_{i=1}^{t-1}\rho^{t-1-i}e_i\right) e_t\\
&=&-\frac\mu cn^\alpha\sum_{t=1}^ne_t+\frac\mu c n^\alpha\sum_{t=1}^n\rho^{t-1}e_t+y_0\sum_{t=1}^n\rho^{t-1}e_t +\sum_{t=1}^n\left(\sum_{i=1}^{t-1}\rho^{t-1-i}e_i\right)e_t\\
&=&\frac\mu cn^{3\alpha/2}\rho^n\sqrt{l(b_n)}V_{23}\{1+o_p(1)\},
\end{eqnarray*}
and
\begin{eqnarray*}
\sum_{t=1}^ny_{t-1}^2=\frac{\mu^2}{2c^3}n^{3\alpha}\rho^{2n}+\frac\mu{c^2}n^{5\alpha/2}\rho^{2n}\sqrt{l(b_n)}V_{22} \{1+o_p(1)\}.
\end{eqnarray*}
Then as $n\to\infty$, we have
\begin{eqnarray*}
  \Delta_1 &=& \frac{\mu^2}{2c^3}n^{3\alpha+1/2}\rho^{2n}\sqrt{l(b_n)}V_{21}\{1+o_p(1)\},\\
  \Delta_2 &=& \frac\mu cn^{3\alpha/2+1}\rho^n\sqrt{l(b_n)}V_{23} \{1+o_p(1)\},\\
  \Delta_3 &=& \frac{\mu^2}{2c^3}n^{3\alpha+1}\rho^{2n}\{1+o_p(1)\}.
\end{eqnarray*}
Therefore, the result holds.
\end{proof}

\section{Concluding remarks}

In this paper, we investigated the limit distribution of the least squares estimator of $(\mu, \rho)$ for the first-order autoregression model whit $\mu \neq 0$. The discussions were took under the assumption that the error variance may be infinite. The existing results fail to hold under this assumption. Our results show that the possible infinite variance affects the convergence rate of the estimator of the intercept in all cases, but only in some cases for the correlation coefficient; see Sections 3.3 and 3.4 for details. Based on the current results, one could build some testing procedures, e.g., $t$-statistics. However, their limit distributions may be quite complex because the least squares estimator has a different limit distribution in different cases, and even is degenerated in the moderate deviations from a unit root cases. Hence, it is interesting to construct some uniform statistical inferential procedures, e.g., confidence region for $(\mu, \rho)^\top$, which are robust to all cases above. Nevertheless, this topic is beyond the scope of the current paper, and will be pursued in the future.

\section*{Acknowledgements}

Xiaohui Liu's research was supported by NSF of China (Grant No.11601197, 11461029), China Postdoctoral Science Foundation funded project (2016M600511, 2017T100475), the Postdoctoral Research Project of Jiangxi (2017KY10),  NSF of Jiangxi Province (No.20171ACB21030).

\end{document}